\documentclass[10pt]{amsart}
\usepackage{amsmath,amssymb, amsthm}
\usepackage{tikz, tikz-cd}
\usepackage{graphics}
\usepackage{stmaryrd}
\usepackage{lipsum}
\usepackage{hyperref}
\usepackage{hyphenat}
\usepackage{enumitem}
\usepackage{ytableau}
\usepackage{mathdots}
\usepackage{newpxtext,newpxmath}

\theoremstyle{plain}

\newtheorem{thm}[]{Theorem}
\newtheorem{ex}[thm]{Example}
\newtheorem{prop}[thm]{Proposition}
\newtheorem{lemma}[thm]{Lemma}
\newtheorem{cor}[thm]{Corollary}
\newtheorem{nt}[thm]{Notation}
\newtheorem*{nt*}{Notation}
\theoremstyle{remark}
\newtheorem{defn}[thm]{Definition}
\newtheorem*{q*}{Question}
\newtheorem{rem}[thm]{Remark}

\newcommand{\bR}{\mathbb{R}}
\newcommand{\bC}{\mathbb{C}}

\newcommand{\bN}{\mathbb{N}}
\newcommand{\bS}{\mathbb{S}}

\newcommand{\iA}{\mathscr{A}}

\newcommand{\iD}{\mathscr{D}}
\newcommand{\iO}{\mathscr{O}}

\newcommand{\Apb}{\iA^{\rm h}}
\newcommand{\Dpb}{\iD^{\rm h}}

\newcommand{\gS}{\mathfrak{S}}
\newcommand{\g}{\mathfrak{g}}
\newcommand{\ggl}{\mathfrak{gl}}
\newcommand{\gsl}{\mathfrak{sl}}

\newcommand{\gln}{\mathfrak{gl}_{n}}

\title[Weight systems which are quantum states]{A note on weight systems which are quantum states}
\author{Carlo Collari}
\address{Dipartimento di Matematica, Universit\`a di Pisa, Largo Bruno Pontecorvo 5, 56127 Pisa, IT}
\email{\tt carlo.collari@dm.unipi.it}
\makeatletter
\@namedef{subjclassname@2020}{\textup{2020} Mathematics Subject Classification}
\makeatother


\let\origmaketitle\maketitle
\def\maketitle{
  \begingroup
  \def\uppercasemath##1{} 
  \origmaketitle
  \endgroup
}
\begin{document}
\begin{abstract}
A result of Corfield, Sati, and Schreiber asserts that  $\gln$-weight systems associated to the defining representation are quantum states. In this short note we extend this result to all  $\gln$-weight systems corresponding to labelling by symmetric and exterior powers of the defining representation.
\end{abstract}

\maketitle
\section{Introduction}

Weight systems and chord diagrams are central objects in the study of finite-type invariants, and Chern-Simons theory, in different contexts.
Universal Vassiliev-type invariants typically take values in a space of diagrams -- e.g.~Jacobi, chord, etc. -- and weight systems are used to recover specific invariants from the universal ones. 
General references for these topics are, for instance, \cite{ChVassiliev, JacMofIVI, Ohtsuki}.

To study finite type invariant for braids, rather than links, a relevant space of diagrams is the space of horizontal chord diagrams.
Recently, horizontal chord diagrams appeared as natural objects in the study of a mathematical framework for quantum physics of branes \cite{corfield2021fundamental, sati2020differential}.
Roughly speaking, an {horizontal chord diagram} on $N$-strands is given by $N$ lines oriented upwards (\emph{strands} or \emph{Wilson lines}), and a number (possibly none) of horizontal segments (\emph{chords}) joining pairs of strands -- cf.~Figure~\ref{fig:horizontal chords}.(a). 
One advantage of horizontal chord diagrams with respect to multi-circle chord diagrams is that they can be endowed with a natural algebraic structure.
The algebra of horizontal chord diagrams $\Apb$ is defined as the algebra generated by (formal) complex linear combinations of horizontal chord diagrams, where the multiplication is given by stacking two diagrams one on top of the other (cf.~Figure~\ref{fig:horizontal chords}.(b)), modulo the infinitesimal pure braid relations (see~Figure~\ref{fig:horizontal chords relations}). Further, we have an anti-linear involution~$\star$~on~$\Apb$ which acts on each horizontal chord diagram by~reversing the direction of the strands, and reflecting the page along an horizontal axis. This involution endows~$\Apb$ with the structure of complex $\star$-algebra (Definition~\ref{def:star-alg}).
In \cite{sati2020differential} (see also~\cite[Section~4]{corfield2021fundamental}), the $\star$-algebra of horizontal chord diagrams~$\Apb$ was interpreted as "higher observables" on certain brane moduli.
This interpretation is coherent with some expected effects in the quantum theory of branes.

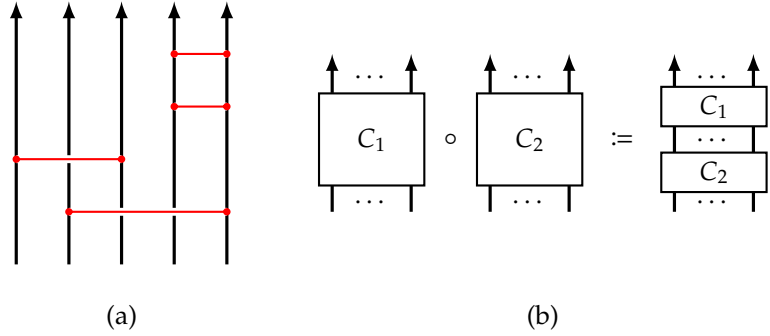
\begin{figure}[]
    \centering
    \begin{tikzpicture}[scale =.35]
        \foreach \x in {-4,-2,0,2,4}{\draw[very thick, -latex] (\x,0) -- (\x, 10);}
        
        \draw[line width=3, white] (-2,2) -- (4, 2);
        \draw[thick, red,fill] (-2,2) circle (.1) -- (4, 2) circle (.1);
        
        \draw[line width=3, white] (-4,4) -- (0, 4);
        \draw[thick, red,fill] (-4,4) circle (.1) -- (0, 4) circle (.1);

        \draw[line width=3, white] (2,6) -- (4, 6);
        \draw[thick, red,fill] (2,6) circle (.1) -- (4, 6) circle (.1);
        
        \draw[line width=3, white] (2,8) -- (4, 8);
        \draw[thick, red,fill] (2,8) circle (.1) -- (4, 8) circle (.1);
        
        \node at (0,-2) {(a)};

        \foreach \x in {8,11}{\draw[very thick, -latex] (\x,2) -- (\x, 8);}
        \foreach \x in {14,17}{\draw[very thick, -latex] (\x,2) -- (\x, 8);}
        \foreach \x in {21,24}{\draw[very thick, -latex] (\x,2) -- (\x, 8);}
        
        \draw[fill,white] (7.5,3) rectangle (11.5,6.5);
        \draw[fill,white] (13.5,3) rectangle (17.5,6.5);
        \draw[fill,white] (20.5,4.25) rectangle (24.5,2.75);
        \draw[fill,white] (20.5,5.25) rectangle (24.5,6.75);
        
        \draw[thick] (7.5,3) rectangle (11.5,6.5);
        \draw[thick] (13.5,3) rectangle (17.5,6.5);
        \draw[thick] (20.5,4.25) rectangle (24.5,2.75);
        \draw[thick] (20.5,5.25) rectangle (24.5,6.75);
        
        \node at (9.5,4.75) {$C_1$};
        \node at (12.5,4.75) {$\circ$};
        \node at (15.5,4.75) {$C_2$};
        \node at (19,4.75) {$\coloneqq$};
        \node at (22.5,3.5) {$C_2$};
        \node at (22.5,6) {$C_1$};

        \node at (22.5,2.25) {$\cdots$};
        \node at (22.5,4.625) {$\cdots$};
        \node at (22.5,7) {$\cdots$};
        
        \node at (9.5,2.25) {$\cdots$};
        \node at (9.5,7) {$\cdots$};
        
        \node at (15.5,2.25) {$\cdots$};
        \node at (15.5,7) {$\cdots$};
        
        \node at (16,-2) {(b)};
    \end{tikzpicture}
    \caption{(a) an horizontal chord diagram on $5$ strands, and  (b) the vertical composition of two horizontal chord diagrams with the same number of strands.}
    \label{fig:horizontal chords}
\end{figure}

Given a $\star$-algebra of observables $\iO$, a \emph{quantum state} (or, simply, \emph{state}) is linear map~$\varphi :\iO \to \bC$ such that $\varphi(x\cdot x^\star) \geq 0$, for all $x\in \iO$, and $\varphi(1_{\iO})>0$.
A \emph{weight system} on horizontal chord diagrams is, by definition, a (complex) linear function from $\Apb$ to $\bC$.
Since horizontal chord diagrams can be interpreted as observables, it is natural to ask the following;
\begin{q*}[{\cite[Question~1.1]{corfield2021fundamental}}]
Which weight systems on horizontal chord diagrams are quantum states?
\end{q*}
This question, which finds its motivation theoretical physics, is also quite interesting from the mathematical viewpoint. In particular, we may restrict ourselves to the case of Lie algebra weight systems associated to a labelling of the Wilson lines by (finite-dimensional) irreducible representations.
These are used to define quantum link invariants from the Kontsevich integral. 

In \cite{corfield2021fundamental}~it was shown that the $\gln$-weight systems associated to the defining representation of $\gln$, which are generators of sorts for all weight systems on $\Apb$, are indeed quantum states. 
The proof of \cite[Theorem~1.2]{corfield2021fundamental} is conceptually simple and exploits a interesting subtle relation between these weight systems and certain distance kernels in Cayley graphs.
The aim of this note is to extend the main result of~\cite{corfield2021fundamental} to $\gln$-weight systems associated to more general sets of labellings.
Recall that an irreducible $\gln$-representation is identified by a Young diagram and a complex number. In this paper we are only concerned with the case when this complex number is $1$. We call these representations \emph{Young diagram representations}.

\begin{thm}\label{thm:main}
All $\gln$-weight systems associated to labelling by symmetric and/or exterior powers of the standard representation are quantum states. 
\end{thm}

It is well known that the symmetric powers of the standard representation are all the irreducible Young diagram representations of~$\mathfrak{gl}_2$. 

\begin{cor}
All $\mathfrak{gl}_2$-weight systems corresponding to irreducible Young diagram representations are quantum states.
\end{cor}

The fundamental representations of~$\mathfrak{gl}_n$ are the given by exterior powers of the standard representation, which gives the next corollary.

\begin{cor}
All $\mathfrak{gl}_n$-weight systems corresponding to any labelling given by fundamental representations are quantum states.
\end{cor}

The paper is organised as follows: in Section~2 we recall some basic definitions and properties of $\star$-algebras, horizontal chord diagrams, representations of $\gln$, and Young symmetrisers. 
In Section~3 we prove our main result. 

\subsection*{Acknowledgements} The author thanks Luigi Caputi, Roberto Castorrini, Sabino Di Trani, Hisham Sati, and Urs Schreiber for the helpful comments and the interesting conversations. This project started when the author was a post-doc at NYU Abu Dhabi.
During the writing of this paper the author partially supported by the MIUR-PRIN project 2017JZ2SW5.
\section{Background material}

In this section we collect some background material concerning $\star$-algebras, representation theory and Lie algebra weight systems.

\subsection{The $\star$-algebra of horizontal chord diagrams}
We start with a definition.

\begin{defn}\label{def:star-alg}
Given a commutative ring $C$ endowed with a ring involution $\overline{\cdot}:C\to C$. 
A \emph{$\star$-algebra}, or \emph{involutive algebra}, over $C$ is a unital associative $C$-algebra $\iO$ together with an involution $\star:\iO \to \iO$, such that: 
\begin{enumerate}[label = (A\arabic*)]
    \item $(1_{\iO})^\star = 1_{\iO}$;
    \item $(z\cdot a+ w\cdot b)^\star = \overline{z} \cdot a^\star + \overline{w}\cdot b^\star$, for all $z,w\in C$ and $a,b\in \iO$;
    \item $(ab)^\star = b^\star a^\star$, for all $a,b\in \iO$.
\end{enumerate}
A \emph{morphism of $\star$-algebras} is a morphism of algebras which commutes with the involution.
\end{defn}

A first class of examples of $\star$-algebras is given by group rings.

\begin{ex}\label{ex:group-star-alg}
Given a group $G$, the group ring $\bC[G]$ has a natural structure of $\star$-algebra given by setting
\[ \left(\sum_{i=1}^{k} z_i \cdot g_i\right)^\star = \sum_{i=1}^{k}\overline{z}_i\cdot g_i^{-1}, \]
for all $z_1,...,z_k\in \bC$ and $g_1,...,g_k\in G$.
\end{ex}

In the next example we present a formal definition of the $\star$-algebra of horizontal chord diagrams $\Apb$. While the reader can keep in mind the pictorial definition given in the introduction, it is useful to have an explicit definition. 

\begin{ex}\label{ex:hcdalg}
An \emph{horizontal chord diagram on $N$ strands} is an element of the free mo\-no\-id~$(\Dpb_N, \circ)$ generated by the pairs $(i,j)$, with $1\leq i < j \leq N$, called \emph{chords}, with neutral element the \emph{chord-less diagram} $\uparrow_N$.
We also consider the \emph{empty chord diagram} $\uparrow_0 = \emptyset$, defined as the diagram with neither chords nor strands, and set $\Dpb_0 = \{ \uparrow_0 \}$.
This monoid has a natural involution $^\star$ given by ``reading the chord diagram backwards'', that is:
\[ [(i_1,j_1)\circ \cdots \circ (i_r,j_r)]^\star = (i_r,j_r)\circ \cdots \circ (i_1,j_1). \]
The \emph{algebra of horizontal chord diagrams on $N$ strands} is given by the quotient algebra
\[ \Apb_{N} \coloneqq \frac{\bC[\Dpb_N]}{I}, \]
where $I\subseteq \bC[\Dpb_N]$ is the ideal generated by elements of the form
\begin{equation}
\tag{\rm 2T}\label{eq:2T}
(i,j)\circ (k,l) - (k,l)\circ (i,j),\qquad 1\leq i < j < k < l \leq N,
\end{equation}
and of the form
\begin{equation}
\tag{\rm 4T}\label{eq:4T}
(i,j)\circ (i,k) + (i,j)\circ (j,k)  - (i,k)\circ (i,j) - (j,k)\circ (i,j),\quad 1\leq i < j < k \leq N,
\end{equation}
see Figure~\ref{fig:horizontal chords relations} for a pictorial representation.
Finally, define the \emph{algebra of horizontal chord diagrams} as $\Apb = \bigoplus_{N\geq 0} \Apb_N$.
Extending $^\star$ by anti-linearity, since $I^\star = I$, the algebra of horizontal chord diagrams acquires the structure of complex $\star$-algebra.~
This involution has actually a deeper and more abstract interpretation as the antipode of an Hopf algebra associated to the homology of some loop spaces -- cf.~\cite[Section~4]{corfield2021fundamental}.
\end{ex}

For small values of $N$ we can explicitly identify $\Apb_N$; we have the following isomorphisms of $\star$-algebras: $\Apb_1  \cong \Apb_0 \cong \bC$, and $\Apb_2 \cong \bC[x]$, where the involution on the latter algebra is defined by setting $x^\star = x$. For $N\geq 3$ the algebra $\Apb_N$ is non commutative. In fact, we have a natural embedding $\Apb_{N-1}\subseteq \Apb_{N}$ and $\Apb_3$ is non-commutative; it is a direct product of the free algebra on two generators $\bC\langle x,y\rangle$  and $\bC[u]$, see~\cite[Proposition~5.11.1]{ChVassiliev}.

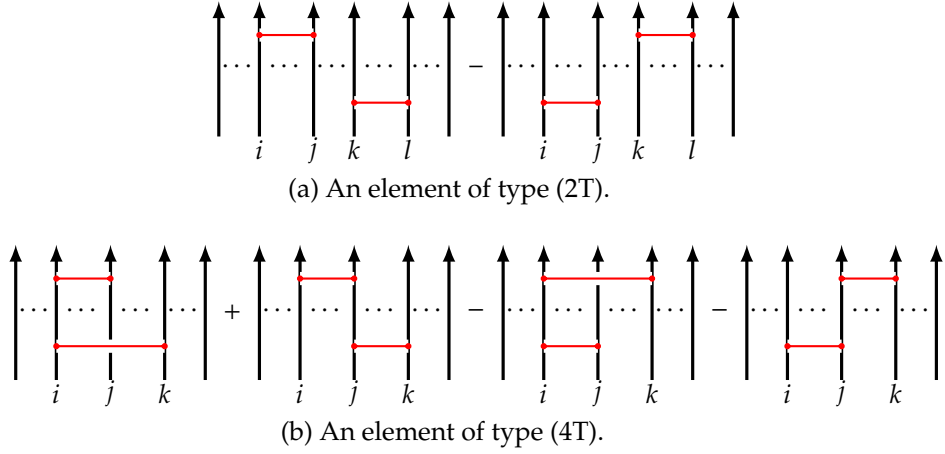
\begin{figure}[]
    \centering
    \begin{tikzpicture}[scale =.18]
    
    \begin{scope}[shift ={+(21,0)}]
      \foreach \x in {-9, -6,-2,1,5, 8}{\draw[very thick, -latex] (\x,0) -- (\x, 10);}
        
        \node at (-7.5,5) {$\cdots$};
        \node at (-4,5) {$\cdots$};
        \node at (-.5,5) {$\cdots$};
        \node at (6.5,5) {$\cdots$};
        \node at (3,5) {$\cdots$};

        \node at (-6,-1) {$i$};
        \node at (-2,-1) {$j$};
        
        \node at (1,-1) {$k$};
        \node at (5,-1) {$l$};
        
        \draw[line width=5, white] (1,7.5) -- (5, 7.5);
        \draw[thick, red,fill] (1,7.5) circle (.15) -- (5, 7.5) circle (.15);

        \draw[line width=5, white] (-2,2.5) -- (-6, 2.5);
        \draw[thick, red,fill] (-2,2.5) circle (.15) -- (-6, 2.5) circle (.15);
        \end{scope}
        \node at (10,5) {--};

        \foreach \x in {-9, -6,-2,1,5, 8}{\draw[very thick, -latex] (\x,0) -- (\x, 10);}
        
        \node at (-7.5,5) {$\cdots$};
        \node at (-4,5) {$\cdots$};
        \node at (-.5,5) {$\cdots$};
        \node at (6.5,5) {$\cdots$};
        \node at (3,5) {$\cdots$};

        \node at (-6,-1) {$i$};
        \node at (-2,-1) {$j$};
        
        \node at (1,-1) {$k$};
        \node at (5,-1) {$l$};
        
        \draw[line width=5, white] (1,2.5) -- (5, 2.5);
        \draw[thick, red,fill] (1,2.5) circle (.15) -- (5, 2.5) circle (.15);

        \draw[line width=5, white] (-2,7.5) -- (-6, 7.5);
        \draw[thick, red,fill] (-2,7.5) circle (.15) -- (-6, 7.5) circle (.15);
        
        \node at (8,-4) {(a) An element of type (2T).};
        \begin{scope}[shift = {+(-13,-18)}]
         \foreach \x in {-11, -8, -4,0,3}{\draw[very thick, -latex] (\x,0) -- (\x, 10);}
        
        \node at (-9.5,5) {$\cdots$};
        \node at (-6,5) {$\cdots$};
        \node at (-2,5) {$\cdots$};
        \node at (1.5,5) {$\cdots$};
        
        \draw[line width=5, white] (-4,7.5) -- (-8, 7.5);
        \draw[thick, red,fill] (-4,7.5) circle (.15) -- (-8, 7.5) circle (.15);

        \draw[line width=5, white] (0,2.5) -- (-8, 2.5);
        \draw[thick, red,fill] (0,2.5) circle (.15) -- (-8, 2.5) circle (.15);

        \node at (-8,-1) {$i$};
        \node at (-4,-1) {$j$};
        
        \node at (0,-1) {$k$};

        \begin{scope}[shift ={+(18,0)}]
       \foreach \x in {-11, -8, -4,0,3}{\draw[very thick, -latex] (\x,0) -- (\x, 10);}
        
        \node at (-9.5,5) {$\cdots$};
        \node at (-6,5) {$\cdots$};
        \node at (-2,5) {$\cdots$};
        \node at (1.5,5) {$\cdots$};
        
        \draw[line width=5, white] (-4,7.5) -- (-8, 7.5);
        \draw[thick, red,fill] (-4,7.5) circle (.15) -- (-8, 7.5) circle (.15);

        \draw[line width=5, white] (0,2.5) -- (-4, 2.5);
        \draw[thick, red,fill] (0,2.5) circle (.15) -- (-4, 2.5) circle (.15);

        \node at (-8,-1) {$i$};
        \node at (-4,-1) {$j$};
        
        \node at (0,-1) {$k$};
        \end{scope}

        \begin{scope}[shift = {+(36,0)}]
         \foreach \x in {-11, -8, -4,0,3}{\draw[very thick, -latex] (\x,0) -- (\x, 10);}
        
        \node at (-9.5,5) {$\cdots$};
        \node at (-6,5) {$\cdots$};
        \node at (-2,5) {$\cdots$};
        \node at (1.5,5) {$\cdots$};
        
        \draw[line width=5, white] (-4,2.5) -- (-8, 2.5);
        \draw[thick, red,fill] (-4,2.5) circle (.15) -- (-8, 2.5) circle (.15);

        \draw[line width=5, white] (0,7.5) -- (-8, 7.5);
        \draw[thick, red,fill] (0,7.5) circle (.15) -- (-8, 7.5) circle (.15);

        \node at (-8,-1) {$i$};
        \node at (-4,-1) {$j$};
        
        \node at (0,-1) {$k$};
        \end{scope}
        
        \begin{scope}[shift ={+(54,0)}]
       \foreach \x in {-11, -8, -4,0,3}{\draw[very thick, -latex] (\x,0) -- (\x, 10);}
        
        \node at (-9.5,5) {$\cdots$};
        \node at (-6,5) {$\cdots$};
        \node at (-2,5) {$\cdots$};
        \node at (1.5,5) {$\cdots$};
        
        \draw[line width=5, white] (-4,2.5) -- (-8, 2.5);
        \draw[thick, red,fill] (-4,2.5) circle (.15) -- (-8, 2.5) circle (.15);

        \draw[line width=5, white] (0,7.5) -- (-4, 7.5);
        \draw[thick, red,fill] (0,7.5) circle (.15) -- (-4, 7.5) circle (.15);

        \node at (-8,-1) {$i$};
        \node at (-4,-1) {$j$};
        
        \node at (0,-1) {$k$};
        \end{scope}

        \node at (5,5) {$+$};
        \node at (23,5) {$-$};
        \node at (41,5) {$-$};
        \node at (20.5,-4) {(b) An element of type (4T).};
        \end{scope}
    \end{tikzpicture}
    \caption{The pictorial representation of the generators of $I\subset \Apb$.}
    \label{fig:horizontal chords relations}
\end{figure}

\subsection{Horizontal chord diagrams, permutations, and standard weight systems}\label{sec:hcd}

To each horizontal chord diagram $C = (i_r,j_r) \circ \cdots \circ (i_1,j_1) \in \Apb_N$ there is a naturally associated permutation:
\begin{equation}\label{def:sigma}
\sigma(C) = \tau_{i_r,j_r} \circ \cdots \circ   \tau_{i_1,j_1} \in \gS_N,    
\end{equation}
where $\tau_{i,j}$ denotes the transposition of $i$ and $j$.
Recall, from Example~\ref{ex:group-star-alg}, that the group ring $\bC[\gS_N]$ has a natural structure of $\star$-algebra.
Indeed, we have the following;

\begin{prop}\label{prop:sigma}
The map $\sigma: \Apb_N \to \bC[\gS_N]: C\mapsto \sigma(C)$ is a morphism of $\star$-algebras.
\end{prop}
 
\begin{proof}
We have to prove that $\sigma$ is a well-defined ring homomorphism. That is we have to show that the images of elements in  Equations~\eqref{eq:2T} and~\eqref{eq:4T} are trivial. 
This is immediate for the elements in Equation~\eqref{eq:2T}.
On the other hand, a direct computation shows that
\[ \sigma((i,j)\circ (i,k)) = \sigma((j,k)\circ (i,j)) = ( j\ k \ i)\]
and that
\[\sigma((i,j)\circ (j,k) ) = \sigma( (i,k)\circ (i,j)) = ( i\ k \ j).\]
Since $\sigma$ commutes with~$\star$, this concludes the proof.
\end{proof}

Following \cite[Section~2.2]{BarNat96}, the  $\gln$ weight system associated to the defining representation, here called \emph{standard $\gln$ weight system}, is the map~$ W_{\tt st} : \Apb \to \mathbb{C} $ which assigns to each chord diagram $C\in \Apb_N$~the integer
\begin{equation}\label{eq:Wst}  W_{\tt st}(C) = w_{\tt st} (\sigma(C)) \coloneqq n^{\text{number of cycles in }\sigma(C)}\ ,\end{equation}
extended by $\mathbb{C}$-linearity. We remark that here the number of cycles in a permutation also includes the trivial cycles (i.e.~the fixed points).

\subsection{Tensor splitting}

Another important operation on $\Apb$ is what we call tensor splitting.
Intuitively, the $\underline{i}$-tensor splitting, with $\underline{i} = (i_1 ,..., i_N)\in \bN^N$, is obtained by replacing the $r$-th strand with $i_r$ parallel strands, and replacing each chord with the sum of all possible ``liftings'' of said chord in the new horizontal chord diagram. More formally, we have the following definition.

\begin{defn}[{\cite[Definition~2.2]{BarNat96}}]\label{def:tensorsplit}
Let $\underline{i} = (i_1 ,..., i_N)$ be a $N$-tuple of positive integers.
The \emph{$\underline{i}$-tensor splitting} is the ring homomorphism
\[ \Delta^{\underline{i}}:\: \Apb_{N} \to \Apb_{\ell_{N+1}(\underline{i})}:\: (j,k) \mapsto \sum_{r=\ell_{j}(\underline{i}) + 1}^{\ell_{j+1}(\underline{i}) }\sum_{s=\ell_{k}(\underline{i}) + 1}^{\ell_{k+1}(\underline{i}) } (r,s)\]
where $\ell_{1}(\underline{i}) = 0$ and $\ell_{r}(\underline{i} ) = i_1 + \cdots +i_{r-1}$, for $0 < r \leq N +1$.
\end{defn}

Let us see in an example the effect of the tensor splitting.
\begin{ex}
Denote $C  = (2,3) \circ(1,3)\in \Apb_3$, we explicitly compute some tensor splittings in this case; for instance, we have that
\[ \Delta^{(1,2,1)}(C) = (2,4) \circ(1,4) + (3,4) \circ(1,4)\in \Apb_4, \]
and also that
\[ \Delta^{(3,1,1)}(C) = (4,5) \circ(1,5) + (4,5) \circ(2,5) + (4,5) \circ(3,5)\in \Apb_5. \]
\end{ex}

Note that $\Delta^{\underline{1}}$, where $\underline{1} = (1,...,1)$, is just the identity map. We conclude this subsection with the following, useful, observation;

\begin{rem}\label{rem:tensor split is morphism}
The tensor splitting $\Delta^{\underline{i}}:\Apb_N\to\Apb_{\ell_{N+1}(\underline{i})}$ is a morphism of $\star$-algebras.
\end{rem}

\subsection{Young symmetrisers}

A \emph{partition of $n\in \mathbb{N}$} is a (non-empty) ordered collection of positive integers $\lambda = (n_1,..., n_k)$, with $n_1 \geq n_2 \geq \dots \geq n_k > 0$, which add up to $n$. If $\lambda$ is a partition of $n$ we shall write $\lambda \vdash n$ or, if it is understood that $\lambda$ is a partition, $|\lambda| = n$.
The (positive) integers $n_1,..., n_k$ are called \emph{parts}, and the number of parts, which is  $k$ in the case at hand, is called the \emph{length} of the partition.

The \emph{Young diagram} associated to a partition  $\lambda = (n_1,..., n_k)$ is a finite collection of boxes arranged in $k$ (left-justified) rows of length (from top to bottom) $n_1,..., n_k$ respectively.
Conversely, to each Young diagram we can associate a partition whose parts are the length of its rows.
We denote both a partition and its associated Young diagram by the same symbol.

\begin{ex}
The partition $\lambda = (5,3,1,1)$ is a partition of\ $10$ whose length is $4$. The corresponding Young diagram $\lambda$ is
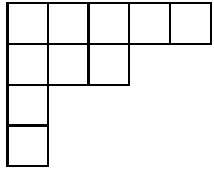
\[\ytableausetup{mathmode, centertableaux}
\begin{ytableau}
\phantom{1}  & \phantom{2} & \phantom{3} & \phantom{4}& \phantom{5}\\ 
\phantom{1} & \phantom{2} & \phantom{3} \\
\phantom{1}  \\
\phantom{1}\\
\end{ytableau}\]
\end{ex}

An essential tool in building irreducible representations of $\gsl_n$, and thus of $\gln$, are Young symmetrisers  -- cf.~\cite[Pages 45-46]{FulHar91}.

\begin{defn}
Given $\lambda \vdash n$, a \emph{Young tableau of shape $\lambda$} is a filling of the Young diagram~$\lambda$ with the numbers $1,...,n$.  
A tableau is \emph{standard} if the number (\emph{label}) on each box is strictly increasing along both rows and columns, \emph{and there are no repeated labels}. 
\end{defn}

Given a Young diagram of shape $\lambda = (n_1, n_2,\ldots, n_k)$, there is a \emph{canonical} (\emph{standard}) \emph{Young tableau} associated to it, which is the following tableau;
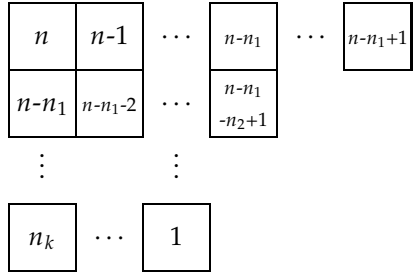
\[\ytableausetup{mathmode, centertableaux,boxsize=2.5em}
\begin{ytableau}
n &  n\text{-}1  & \none[\cdots] & \scriptstyle n\text{-}n_1 & \none[\cdots] & \scriptstyle n\text{-}n_1\text{+}1\\ 
n\text{-}n_1 & \scriptstyle n\text{-}n_1\text{-}2 &   \none[\cdots] &  \begin{array}{c}\scriptstyle{n\text{-}n_1} \\ \scriptstyle{\text{-}n_2 \text{+1}}\\\end{array}\\
\none[\vdots] & \none[] &  \none[\vdots] \\
n_k& \none[\cdots] & 1\\
\end{ytableau}\]
To get hold of how the canonical tableau is defined, we give a concrete example.

\begin{ex}
The canonical Young tableau of shape $\lambda = (5,3,1,1)$ is the following
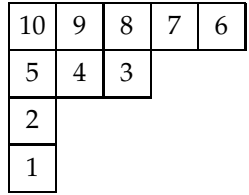
\[\ytableausetup{mathmode, centertableaux, boxsize=1.75em}
\begin{ytableau}
10 & 9 & 8 & 7 & 6\\ 
5 & 4 & 3 \\
2 \\
1\\
\end{ytableau}\]
\end{ex}

Let $T^\lambda$ be the set of all tableaux of shape $\lambda=(n_1,...,n_k)$. 
A permutation~$\sigma\in \gS_{|\lambda|}$ acts on each tableau~$t^\lambda\in T^\lambda$ by permuting the labels, denote by $\sigma.t^\lambda$ the resulting tableau. 
Fixed a tableau $t^\lambda$ we can define two sub-groups of $\gS_{|\lambda|}$; that is, the \emph{row stabiliser}
\[  \mathfrak{R}_{t^{\lambda}} = \left\{\  \sigma \in \gS_{|\lambda|}\ \left\vert\  \begin{array}{c}\text{the set of labels on corresponding} \\ \text{rows of }  \sigma.t^\lambda \text{ and } t^\lambda  \text{ are the same}\end{array} \right.\right\},\]
and the \emph{column stabiliser}
\[  \mathfrak{C}_{t^{\lambda}} =  \left\{\  \sigma \in \gS_{|\lambda|}\ \left\vert\  \begin{array}{c}\text{the set of labels on corresponding} \\ \text{columns of }  \sigma.t^\lambda \text{ and } t^\lambda  \text{ are the same}\end{array} \right.\right\}.\]
Using these sub-groups we can associate to each tableau its Young symmetriser.

\begin{defn}
The \emph{unnormalised Young symmetriser} $\tilde{c}_{t^{\lambda}}\in \mathbb{C}[\gS_{|\lambda|}]$ associated to the tableau $t^\lambda$ is defined as $\tilde{c}_{t^{\lambda}} :=  a_{t^{\lambda}} b_{t^{\lambda}}$, where
\[ a_{t^{\lambda}} = \sum_{\sigma\in \mathfrak{R}_{t^{\lambda}}} \sigma\qquad \qquad \qquad b_{t^{\lambda}} = \sum_{\sigma \in \mathfrak{C}_{t^{\lambda}}} \rm{sign}(\sigma)\sigma.\]
The \emph{Young symmetriser}  $ c_{t^{\lambda}}\in \mathbb{C}[\gS_{|\lambda|}]$ is a re-scaling of $\tilde{c}_{t^{\lambda}}$ by a positive rational number, in such a way that ${c}_{t^{\lambda}} {c}_{t^{\lambda}} = {c}_{t^{\lambda}}$ -- cf.~\cite[Lemma~4.6]{FulHar91}.
\end{defn}

\begin{nt}
The Young symmetriser associated to a partition $\lambda$, and not to a tableau $t^{\lambda}$, is the symmetriser associated to the canonical tableau, and we write $c_\lambda$ \emph{en lieu} of~$c_{t^{\lambda}}$.
\end{nt}

\begin{rem}\label{rem:expansion_symmetriser}
Note that if $\sigma$ appears in $a_\lambda$ (or $b_\lambda$) then also $\sigma^{-1}$ does, and they appear with the same (real) coefficients. 
The same does not hold for $c_\lambda$ (see the example below).
Nonetheless, if $\lambda$ is either of the form $(n)$ or $(1,...,1)$, then  $c_\lambda = c_\lambda^\star$ as elements of $\bC[\gS_{\vert \lambda \vert}]$.
\end{rem}

\begin{ex}\label{ex:symmetrisers}
Consider the partition $\lambda\vdash 3$ given by $(2,1)$, the corresponding Young diagram $\lambda$ is the following
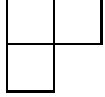
\[\ytableausetup{mathmode, centertableaux}
\begin{ytableau}
\phantom{1}  & \phantom{2}\\ 
\phantom{2} \\
\end{ytableau}\]
Two examples of Young tableau of shape $\lambda$ are
\[t^{\lambda} = \ytableausetup{mathmode, centertableaux}
\begin{ytableau}
1  & 2 \\ 
3\\
\end{ytableau}\qquad\qquad s^\lambda( = (1,3).t^\lambda) = \ytableausetup{mathmode, centertableaux}
\begin{ytableau}
3  & 2 \\ 
1\\
\end{ytableau}\]
Observe that $s^\lambda$ is standard (actually, it is the canonical tableau), whereas $t^\lambda$ is not.
In these cases we have:
\[ \mathfrak{R}_{t^{\lambda}} = \{ {\rm id}, (1,2) \},\qquad \mathfrak{C}_{t^{\lambda}} = \{ {\rm id}, (1,3) \}, \]
and
\[ \mathfrak{R}_{s^\lambda} = \{ {\rm id}, (2,3) \},\qquad \mathfrak{C}_{s^\lambda} = \{ {\rm id}, (1,3) \}, \]
respectively. It follows that:
\[ \tilde{c}_{t^{\lambda}} = ({\rm id} + (1,2))({\rm id} - (1,3)) ={\rm id} + (1,2) - (1,3) - (1,3,2),\] 
and
\[ \tilde{c}_{s^\lambda} = ({\rm id} + (2,3))({\rm id} - (1,3)) = {\rm id} + (1,2) - (1,3) - (1,2,3).\]
\end{ex}

The next subsection is dedicated to the description of irreducible $\gln$-rep\-re\-sen\-ta\-tions in terms of partitions. Our main scope there is to set some notation, for the general theory the reader may refer to \cite[Chapter 15]{FulHar91}.

\subsection{Representations of $\gln$}

Denote by $\rho_{\tt st}$ the defining representation of~$\gln$, that is the natural action of $\gln = {\rm End}(\bC^n)$ on $\bC^n$, and denote by $\rho^{\gsl}_{\tt st}$ the defining representation of $\gsl_n$ -- which is the restriction of $\rho_{\tt st}$ to $\gsl_n$.

It well-know that (finite-dimensional) irreducible representation of $\gsl_n$ are associated to a partition of length at most $n$ \cite[\S 15.3]{FulHar91}. The representation corresponding to a partition $\lambda$ can be described explicitly.
Consider the natural action (by permutation of the tensor factors) of the symmetric group $\gS_k$ on the tensor product of~$k$ copies of $\bC^n$. This action commutes with the action of $\gsl_n$, given by $(\rho_{\tt st}^{\gsl})^{\otimes k}$. In particular, we can consider the images of the action of a Young symmetriser~$c_{t^{\lambda}}\in \bC[\gS_k]$ (with $|\lambda| = k$) in  $\bC^{n} \otimes \dots \otimes \bC^{n}$.

\begin{defn}
Let $\lambda\vdash k$, and $n\geq k$. The image of the action of $c_\lambda$ on $(\bC^{n})^{\otimes k}$  is called the \emph{Weyl module} corresponding to $\lambda$, and is denoted by $\bS_\lambda(\bC^n)$.
\end{defn}

The restriction of $(\rho_{\tt st}^{\gsl})^{\otimes k}$ to the Weyl module $\bS_\lambda(\bC^n)$ associated to $\lambda\vdash k$ is isomorphic to the irreducible representation of $\gsl_n$ with highest weight $w_{\lambda}$, denoted by $\rho^{\gsl}_\lambda$.
We can extend $\rho^{\gsl}_\lambda$ to a representation of $\gln =\gsl_n\oplus \mathbb{C}\langle {\rm Id}_{\bC} \rangle$ by making~$ {\rm I}_n$ act trivially on $V_\lambda$. For $z\in\mathbb{C}$, consider the representation $\rho_{z}:\gln \to {\rm End}(\bC)$ defined by
\[ \rho_{z}(g + k {\rm Id}_{\bC})[x]  = zk x,\quad \forall g\in \gsl_n,\ k, x\in\bC\, . \]
Each irreducible representation of $\gsl_n\oplus \mathbb{C}\langle  {\rm Id}_{\bC} \rangle = \ggl_n$ is isomorphic to a tensor product of the form~$\rho_\lambda \otimes \rho_{z}$ -- see \cite[\S~15.5]{FulHar91} -- where the tensor product of two representations, say $\rho$ and $\rho'$, is given by $$(\rho \otimes \rho')(g)[v_1\otimes v_2] = \rho(g)[v_1] \otimes v_2 + v_1 \otimes  \rho'(g)[v_2] .$$ 
We refer to the Young diagram representation $\rho_\lambda \otimes \rho_1$ simply as~$\lambda$. 
The symmetric and exterior power of the defining representation are associated to the partitions of the form $(k)$ and $(1,...,1)$, respectively.

\subsection{General Lie algebra weight systems}
In this subsection we recall the general construction of Lie algebra weight system, and then we specialise this construction to the case at hand -- see~\cite[Ch.~6]{ChVassiliev}, \cite[Ch.~14]{JacMofIVI}, \cite[Sec.~2]{BarNat96}, and references therein for a comprehensive overview.

The main ingredients in the definitions of Lie algebra weight system on $\Apb_N$ are
\begin{enumerate}
\item a (finite-dimensional complex) Lie algebra $\g$;
\item an ${\rm ad}$-invariant non-degenerate bi-linear form $\langle \cdot, \cdot \rangle$ on $\g$; 
\item an ordered collection of finite-dimensional $\g$-representations $\underline{\rho} = (\rho_1, ..., \rho_N)$, called \emph{labelling} where $\rho_i:\g \to {\rm End}(V_i)$ for each $i=1,...,N$.
\end{enumerate}
The basic idea is to associate to each horizontal chord diagram $C\in \Apb_N$ an element in ${\rm End}(V_1 \otimes \cdots \otimes V_N)$, and then take the trace to obtain a complex number.

Firstly, fix an orthonormal basis (with respect to $\langle \cdot , \cdot \rangle$) for $\g$, say $e_1,...,e_d$. This choice, and the fact that $\langle \cdot , \cdot \rangle$ is non-degenerate, allow us to identify the aforementioned bi-linear form with the element 
$$\sum_{i=1}^{{\rm dim}(\g)} e_i  \otimes e_i  \in \g \otimes \g .$$
We can decompose, by creating a unique minima on each chord, each horizontal chord diagram in local pieces like the ones shown in Figure~\ref{fig:dictionary}, which correspond to the illustrated maps. Interpreting horizontal juxtaposition as tensor product, and vertical composition as the usual composition, we obtain the desired map. (Crossings between chords and Wilson lines have no meaning.) 
\begin{figure}
    \centering
    \begin{tikzpicture}[xscale = .75]

    \begin{scope}[shift = {+(-1.5,0)}]
    \draw[ very thick, -latex] (0,0) node[below] {$V_i$} -- (0,2) node[above] {$V_i$};
    \draw[fill, red] (0,1) circle (.075);
    \draw[thick, red] (0,1) -- (1,0) node[below] {$\g$};
    \node[below] at (.5,0) {$\otimes$};
    \node[below] at (.5,-1) {$\rho_i$};
    \end{scope}
    \begin{scope}[shift = {+(2,0)}]
    \draw[ very thick, -latex] (0,0) node[below] {$V_i$} -- (0,2) node[above] {$V_i$};
    \draw[fill, red] (0,1) circle (.075);
    \draw[thick, red] (0,1) -- (-1,0) node[below] {$\g$};
    \node[below] at (-.5,0) {$\otimes$};
    \node[below] at (-.5,-1) {$\rho_i$};
    \end{scope}

    \draw[ very thick, -latex] (-3.5,0) node[below] {$V_i$} -- (-3.5,2) node[above] {$V_i$};
    \node[below] at (-3.5,-1) {${\rm Id}_{V_i}$};
    
    \draw[thick, red] (3,1) node[above] {$\g$} .. controls +(.5,-.5) and +(-.5,-.5) .. (5,1) node[above] {$\g$};
    \node[below] at (4.25,-1) {$1 \mapsto  \langle \cdot,\cdot \rangle$};
    \node[above, red] at (4 ,1) {$\otimes$};
    \node[below,red] at (4 ,0) {$\bC$};

    \end{tikzpicture}
    \caption{Dictionary between graphical and abstract construction of Lie algebra weight systems.}
    \label{fig:dictionary}
\end{figure}
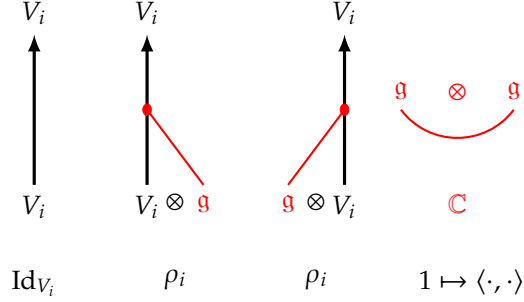
More explicitly, to each chord $C_{i,j}^{N} = [(i,j)]\in \Apb_N$ we are associating the element
\[ \widetilde{W}_{\underline{\rho}}(C_{i,j}^{N}) = \sum_{r=1}^{{\rm dim}(\g)} {\rm Id}_{V_1}\otimes \cdots
\otimes \overset{\text{$i$-th pos.}}{\rho_i(e_r)}
\otimes \cdots
\otimes \overset{\text{$j$-th pos.}}{\rho_j(e_r)}\otimes \cdots \otimes  {\rm Id}_{V_N}\, .\]
It can be shown that $\widetilde{W}_{{\underline{\rho}}}$ induces a well-defined morphism of $\star$-algebras between $\Apb_N$ and ${\rm End}(V^{\otimes N})$, cf.~\cite[Sec.~2.1]{BarNat96}, \cite[Sec.~14.2-.3]{JacMofIVI}. 
We can now give the definition of Lie algebra weight system.

\begin{defn}
Given a metrised Lie algebra $(\g,\langle \cdot , \cdot \rangle )$ and an ordered collection $\underline{\rho} = (\rho_1, ..., \rho_N)$ of $\g$-representations, the corresponding \emph{Lie algebra weight system} is given by
\[ W_{\underline{\rho}} (C)  = {\rm Tr}(\widetilde{W}_{\underline{\rho}} (C)), \]
for each $C\in \Apb_N$.
\end{defn}

\begin{rem}\label{rem:phi(1)}
For each $\g$-weight system $W_{\underline{\rho}}(\uparrow_{N}) = {\rm Tr} ({\rm Id}_{V_1 \otimes \cdots \otimes V_N})=  \prod_i {\rm dim}(V_i) > 0$.
\end{rem}

The $\gln$-weight system $W_{\tt st}$, associated to the bi-linear form $\langle A , B \rangle = {\rm Tr}( AB)$ and to the defining representation, admits a simple combinatorial description in terms of permutations, as recalled in Equation~\eqref{eq:Wst} -- cf.~\cite{BarNat96}.
There exists a similar description for more general $\gln$-weight system -- compare with~\cite{Nag03}.
\begin{defn}
Let $i_{0}\in \{1,...,N\}$, $0\leq k \leq N-i_0$, and $\lambda$ be a partition of $k$.
The \emph{small Young symmetriser} $c_{(N,i_0,\lambda)}\in \bC[\gS_N]$ is the image of the Young symmetriser $c_\lambda\in \bC [\gS_k]$ via the inclusion
\[ \bC[\gS_k] \overset{\cong}{\longrightarrow} \bC[\gS_{\{i_0 + 1 ,..., i_0 + k\}}] \subset \bC[\gS_N]\ ,\]
where $\gS_{\{i_0 + 1 ,..., i_0 + k\}}$ denotes the group of permutations of the set $\{i_0 + 1 ,..., i_0 + k\}$.
Given a vector of $\gln$-representations $\underline{\rho} = (\lambda^1,...,\lambda^N)$, we can define the 
\emph{Young symmetriser associated to $\underline{\rho}$} as the product
\[ c_{\underline{\rho}}\coloneqq \prod_{i+1}^{n} c_{(N,\lambda^i, \sum_{j< i} \vert \lambda^j\vert )}   \in \bR [\gS_{\vert {\underline{\rho}}\vert}]\subset \bC [\gS_{\vert {\underline{\rho}}\vert}]\ . \]
\end{defn}
The weight system $W_{\underline{\rho}}$ can be also defined as the following composition:
\[ \Apb_N \overset{\Delta^{(\vert \lambda^1\vert ,..., \vert \lambda^N\vert )}}{\longrightarrow} \Apb_{\vert {\underline{\rho}}\vert} \overset{\sigma}{\longrightarrow} \bC[\gS_{\vert {\underline{\rho}}\vert}] \overset{c_{\underline{\rho}}}{\longrightarrow} \bC[\gS_{\vert {\underline{\rho}}\vert}] \overset{w_{\tt st}}{\longrightarrow} \mathbb{C}\ ,\]
where $\vert {\underline{\rho}}\vert \coloneqq \sum_{j \leq N} \vert \lambda^j\vert$ -- cf.~\cite{Nag03, NagTak07} . Roughly speaking, we are writing each $\lambda^i$ as the composition of the representation $\rho_{\tt st}^{\otimes \vert \lambdaî\vert}$ -- encoded by $\Delta^{(\vert \lambda^1\vert ,..., \vert \lambda^N\vert )}$, cf.~\cite[Def.~2.2]{BarNat96} -- and the projection onto $\bS_{\lambda}(\bC)$ -- encoded by the multiplication by $c_{\underline{\rho}}$ -- and taking the trace by applying $w_{\tt st}$.

\section{Proof of the main result}

In \cite{corfield2021fundamental} it was proven that the standard $\gln$ weight system is a quantum state. The proof goes as follows: first, one observes that $W_{\tt st}$ factors as
\begin{equation*}
\begin{tikzcd}[row sep=1.5em, column sep = 1.5em]
\Apb_N \arrow[d, "\sigma"']\arrow[r, "W_{\tt st}" ] & \bC\\
\bC[\gS_N] \arrow[ur, "w_{\tt st}"' ]
 \end{tikzcd}
\end{equation*}
where $\sigma$ is the morphism of $\star$-algebras defined in Subsection~\ref{sec:hcd}.
The key observation made in \cite{corfield2021fundamental} is that $w_{\tt st}$ is the kernel of the Cayley distance function, which is a class functions (i.e.~it is constant on conjugacy classes). Eigenvalues and eigenvectors of class functions on the symmetric group are known --  cf.~\cite[Section~3]{corfield2021fundamental} and references therein. So one concludes by observing that these eigenvalues are non-negative for all possible $n$. This proves that $w_{\tt st}$ is a quantum state, and the main result of \cite{corfield2021fundamental} is now a consequence of the following easy lemma.
\begin{lemma}\label{lem:morphism_and_states}
Let $\iO$ and $\iO'$ be $\star$-algebras.
If $\phi: \iO' \to \bC$ is a state, and $f: \iO \to \iO'$ is a morphism of $\star$-algebras, then~$\phi \circ f$ is a state.
\end{lemma}

We want to use a similar reasoning to prove our main result. We know that~$W_{\underline{\rho}}$ factors as the following composition:
\[ \Apb_N \overset{\Delta_{\vert {\underline{\rho}}\vert}}{\longrightarrow} \Apb_{\vert {\underline{\rho}}\vert} \overset{\sigma}{\longrightarrow} \bC[\gS_{\vert {\underline{\rho}}\vert}] \overset{c_{\underline{\rho}}}{\longrightarrow} \bC[\gS_{\vert {\underline{\rho}}\vert}] \overset{w_{\tt st}}{\longrightarrow} \mathbb{C}\ .\]
While the map $\sigma \circ \Delta_{\vert \underline{\rho}\vert}$ is a morphism of $\star$-algebras, the map $c_{\underline{\rho}}$ is not -- for any choice of ${\underline{\rho}}$ which is not ${\tt st}$. Therefore, we cannot conclude using  Lemma~\ref{lem:morphism_and_states}.
Nonetheless, we can show that the composition $w_{\tt st} \circ c_{{\underline{\rho}}}$ is a state for suitable choices of ${\underline{\rho}}$, proving our main result. 

\begin{defn}
Let $\iO$ be a $\star$-algebra, and $x\in \iO$.~A linear function $\phi: \iO \to  \bC$ is \emph{class-like with respect to $x$} if $\phi(xy) = \phi(y x)$, for each $y\in\iO$.
\end{defn}

Linear functions arising from class functions are class-like functions

\begin{ex}\label{ex:class-like}
If $f: \gS_m \to \bC$ is a class function (i.e.~it is constant in each conjugacy class), then
\[\phi\left(\sum_{i} z_i \sigma_i\right) \coloneqq \sum_{i} z_i f(\sigma_i)\ ,\]
is class-like for all $c$. 
In fact, given $c = \sum_{j} x_j \eta_j $ and $x = \sum_{i} z_i \sigma_i$, we have:
\[\phi\left(c x\right) = \sum_{i, j} x_jz_i f(\eta_j\sigma_i)  =  \sum_{i, j} z_i x_jf(\sigma_i\eta_j) = \phi\left( x c\right)\ .\]
\end{ex}

States which are class-like with respect to special elements can be used to produce new states, as shown in the following lemma.

\begin{lemma}\label{lem:class-like_and_idempotents}
Let $c \in \iO$ be  such that $c^2 = \alpha c$, $\alpha \in \bR_{> 0}$, and  $c^\star = c$. If a state $\phi: \iO \to  \bC$ is class-like with respect to $c$, then $\psi = \phi \circ (c \cdot)$ is also a state.
\end{lemma}
\begin{proof} Let $a \in \iO$, then
\[\psi (a a^{\star}) = \tfrac{1}{\alpha}\phi (c^2 a a^{\star}) = \tfrac{1}{\alpha}\phi ( c a a^{\star}c) = \tfrac{1}{\alpha}\phi ( c a a^{\star}c^\star) =  \tfrac{1}{\alpha}\phi (c a (ca)^{\star})  \geq 0.\]
\end{proof}

We remark that the hypothesis $c = c^\star$ plays a key role in the above lemma. Now, we are ready to prove the main theorem.

\begin{proof}[Proof of Theorem~\ref{thm:main}]
Let ${\underline{\rho}}= ( \lambda^1 ,..., \lambda^N)$ be a labelling such that each $\lambda_i$ is either a symmetric or an exterior power of the standard representation. 
Since $W_{\underline{\rho}}(\uparrow_N)>0$, see Remark~\ref{rem:tensor split is morphism}, we only have to prove that $W_{\underline{\rho}}(CC^\star)\geq 0$, for each $C\in\Apb_N$.
We recall that the following triangle
\begin{equation*}
\begin{tikzcd}[row sep=1.5em, column sep = 1.5em]
\Apb_N \arrow[d, "\sigma\circ \Delta^{(\vert \lambda^1\vert ,..., \vert \lambda^N\vert )}"']\arrow[r, "W_{\tt st}" ] & \bC\\
\bC[\gS_{\vert \underline{\rho}\vert}] \arrow[ur, "w_{\tt st}\circ (c_{\underline{\rho}}\cdot)"' ]
 \end{tikzcd}
\end{equation*}
is commutative.
From Proposition~\ref{prop:sigma} it follows that $\sigma \circ \Delta_{\rho}$ is a morphism of $\star$-algebras. Thence, by Lemma~\ref{lem:morphism_and_states}, in order to show that $W_{\underline{\rho}}$ is a state it is sufficient to prove that $w_{\tt st}\circ (c_{\underline{\rho}}\cdot)$ is a state. 
Note that $c_{\underline{\rho}}$ is idempotent since small Young symmetrisers are idempotent and their actions commute (as they have disjoint support).
Furthermore, since each $\lambda^i$ is either a symmetric or exterior power of the defining representation, , the commutativity of the small Young symmetriser and Remark~\ref{rem:expansion_symmetriser}, ensure that $c_{\underline{\rho}} = c_{\underline{\rho}}^\star$.
Since $w_{\tt st}$ is state-like with respect to $c_{\underline{\rho}}$ (cf.~Example~\ref{ex:class-like}) the statement follows directly from Lemma~\ref{lem:class-like_and_idempotents}
\end{proof}

\bibliographystyle{plain}
\bibliography{references}

\end{document}